\documentclass[11pt]{amsart}
\usepackage[active]{srcltx}
\usepackage{mathrsfs}
\usepackage[T1]{fontenc}

\usepackage{latexsym,enumerate}
\usepackage{amsmath,amssymb,amsthm,amsfonts,latexsym}
\usepackage{graphicx}
\usepackage[bookmarks]{hyperref}

\hypersetup{backref, colorlinks=true}
\hypersetup{backref, colorlinks=true, colorlinks   = true,
	urlcolor     = blue, linkcolor = blue, citecolor   = magenta
}
\def\adh#1{\overline{#1}}

\let\=\partial

\newtheorem {pro}{Proposition}[section]
\newtheorem {thm}[pro]{Theorem}%[section]
\newtheorem {cor}[pro]{Corollary}%[section]
\newtheorem{lem}[pro]{Lemma}

\theoremstyle{definition}
 \newtheorem {rem}[pro]{Remark}%[section]
\newtheorem {dfn}[pro]{Definition}%[section]

\newtheorem {exa}[pro]{Example}

\newcommand{\nbf}{{\mathbf{n}}}

\newcommand{\s}{\mathcal{S}}

\newcommand{\tra}{\mathbf{tr}}

\newcommand{\R}{\mathbb{R}}

\newcommand{\N}{\mathbb{N}}
\newcommand{\cc}{\mathscr{C}}

\newcommand{\et}{\quad \mbox{and} \quad }

\newcommand{\hn}{\mathcal{H}}

\newcommand{\mba}{ {\overline{M}}}

\newcommand{\omd}{{\Omega}}

\newcommand{\ep}{\varepsilon}

\newcommand{\pa}{\partial}

\newcommand{\hh}{\mathcal{V}}

\newcommand{\bou}{\mathbf{B}}

\newcommand{\supp}{\mbox{\rm supp}}
\newcommand{\xo}{{x_0}}
\title[]{ Density of Neumann regular smooth functions in Sobolev spaces of subanalytic manifolds}

\makeatletter
 \thanks{Research partially supported
by the NCN grant  2021/43/B/ST1/02359.}
\@namedef{subjclassname@2020}{%
\textup{2020} Mathematics Subject Classification}
\makeatletter

\@addtoreset{equation}{section}

\makeatother

\author[ G. Valette]{ Guillaume Valette}

\address[G. Valette]{Instytut Matematyki Uniwersytetu
Jagiello\'nskiego, ul. S. \L ojasiewicza 6, Krak\'ow, Poland}\email{guillaume.valette@im.uj.edu.pl}



\begin{document} \maketitle
\begin{abstract}
We give characterizations of the bounded subanalytic $\mathscr{C}^\infty$ submanifolds $M$ of $\mathbb{R}^n$ for which the space of Neumann regular functions is dense in Sobolev spaces.  
	By ``Neumann regular function'', we mean a function which is smooth at almost every boundary point and whose gradient is tangent to the boundary. In the case $p\in [1,2]$, we prove that the Neumann regular  elements of $\mathscr{C}^\infty(\overline{M})$  are dense in $W^{1,p}(M)$ if and only if $M$ is connected at almost every boundary point. In the case $p$ large, we show that the Neumann regular Lipschitz elements of $\mathscr{C}^\infty(M)$  are dense in $W^{1,p}(M)$ if and only if $M$ is connected at every boundary point. The proof  involves the construction of Lipschitz Neumann regular partitions of unity, which is of independent interest.
	\end{abstract}
	\section{Introduction}
Recent works on Sobolev spaces of subanalytic domains and manifolds \cite{poincfried, poincwirt,trace,lprime} point out that, in spite of the fact that these sets may admit non metrically conical singularities, they constitute a nice category to investigate partial differential equations. Let us recall that the subanalytic category comprises all the sets defined by equalities and  inequalities on analytic functions (see section \ref{sect_sub}), offering a level of generality which is valuable for applications.

Applications of the aforementioned works to PDEs were given in \cite{gupel} where it was shown that the classical weak formulations of the Laplace equation under Dirichlet or Neumann (or mixed) condition go over the framework of subanalytic domains, not necessarily Lipschitz. It is sometimes useful to give weak formulations only involving Neumann regular functions, which requires to show the density of these functions in $W^{1,2}(\omd)$, for a given  domain $\omd$. This fact, which was established on smooth domains and on  Lipschitz  simplicial complexes  \cite{droniou}, is not true on any Lipschitz domain (see \cite{droniou}).  Theorem \ref{thm_p_petit} of this article entails that this density holds  on all subanalytic Lipschitz bounded domains, and  more generally, given a subanalytic bounded $\cc^\infty$ submanifold $M$ of $\R^n$:
	\begin{cor}\label{cor_2}
For $p\in [1,2]$,   the space
 $$\{u\in \cc^\infty (\mba) : u \mbox{ is Neumann regular}  \}$$
 is dense in $W^{1,p}(M)$ if and only if $M$ is connected at $\hn^{m-1}$ almost every boundary point.
\end{cor}
By ``connected at a boundary point $\xo$'', we   mean that $\bou(\xo,\ep)\cap M$ is connected for $\ep>0$ small.  Since we put no regularity assumption except subanalicity, this set may consist of more than one (but finitely many) connected components, like for instance  if $M=\{(x,y,z)\in  (0,1)^3: x^2+y^2<z^2\}$ and $\xo=(0,0,0)$. Lipschitz domains are clearly connected at every boundary point.
 Theorem \ref{thm_p_petit}  provides a more general result,  showing the density for all $p$ not greater than the codimension of the singularities of the frontier (which is always at least $2$). In particular, when the boundary of a domain of $\R^3$ has isolated singularities, this density holds up to $p=3$. This theorem also yields the density of  Neumann regular smooth functions vanishing near a subset $Z$ of the boundary in the space of functions having trace $0$ on $Z$, generalizing other results of \cite{droniou} that are useful to handle weak formulations of PDEs under mixed Dirichlet-Neumann conditions.

%It seems that one could spare the connectedness assumption by replacing $\cc^\infty(\mba)$ with $W^{1,\infty}(M)$. The above theorem however provides more information in many cases. For instance, when $\dim \delta M\setminus \pa M=0$ then this is true for $W^{1,p}(M)$, $p\in [1,3]$.
% We can spare, which identifies the respective Sobolev spaces.
Example \ref{exa_petit} shows that these results cannot be generalized to all $p$, as density of smooth functions itself fails. In the case ``$p$ large'', where  the density  of $\cc^\infty(\mba)$ holds (Theorem \ref{thm_trace}, again up to connectedness at frontier points),
we  establish the density of Lipschitz $\cc^\infty$ Neumann regular functions (Theorem \ref{d}).  We construct for this purpose Lipschitz Neumann regular partitions of unity (Proposition \ref{pro_p}), which is of its own interest as it can be
useful to localize weak formulations of PDEs.
We give a counterexample to show that $\cc^2(\mba)$ Neumann regular functions may fail to be dense in $W^{1,p}(M)$ for arbitrarily large values of $p$ even if $M$ is a Lipschitz subanalytic domain of $\R^n$. This example also shows that the Neumann regular partitions of unity given by Proposition \ref{pro_p} could not be required to be restrictions of $\cc^2$ functions on $\R^n$.

%Theorem \ref{thm_trace} showing that density of smooth functions holds for $p$ large (when the manifold is connected at boundary points), it is therefore natural to wonder if Neumann regular functions are also  dense for $p$ large. We address this issue in the next section.

	\section{Some notations and conventions}\label{sect_notations}
	Throughout this article, the letter $M$ stands for a bounded subanalytic $\cc^\infty$ submanifold of $\R^n$. We will use the following notations.

	\begin{itemize}
	 \item 	 $\bou(x,\ep)$, open  ball in $\R^n$  of radius $\ep$   centered  at $x\in \R^n$ (for the  euclidean distance).

% 		\item $\sph(x,\ep)$, sphere    of radius  $\ep>0$ centered at   $x\in  \R^n$. As customary, $\sph(0_{\R^n},1)$ will however be denoted $\sph^{n-1}$, $0_{\R^n}$ being the origin of $\R^n$.

		 	\item 	$x\cdot y$  and $|x|$,  euclidean inner product of $x$ and $y$ and euclidean norm of $x$.

		 \item 	$\adh A$, closure of a set $A\subset \R^n$. We also set $\delta A:= \adh A \setminus A$.
		 \item $\pa M$, Lipschitz regular locus of $M$ (section \ref{sect_sub}).
		 \item $\Gamma$, smooth part of $\pa M$ (section \ref{sect_sub}).

	\item $D_xh$   derivative of a mapping $h$ at a point $x$.  As customary, we however write $h'(x)$ whenever $h$ is a one-variable function.

% 	\item 	$r \cdot A$, image of $A\subset \R^n$ under the mapping $x\mapsto rx$, if $r\in \R$.

% 		\item $<u,v>$, $L^2$-inner product of $\omd$, i.e. $\int_\omd uv$. %We write $<u,v>_U$ for the $L^2$-inner product of $U$, if $U$ is a manifold

	\item 	$\nabla u$, gradient of $u:M \to \R$.

	\item 	Given   $p\in [1,\infty)$, $W^{1,p}(M)$ stands for the Sobolev space of $M$, i.e. $$W^{1,p}(M):= \{u\in L^p(M),\; |\nabla u| \in L^p(M)\},$$ and, given an open subset $Z$ of $\pa M$,
	$$W^{1,p}(M,Z):= \{u\in W^{1,p}(M), \tra_{\pa M} u=0\; \mbox{on $Z$}\},$$
	where $\tra_{\pa M} $ is the trace operator defined in section \ref{sect_tra}.
%   We denote by $\supp\, u$ the support of a distribution $u$ on $M$, and, given an open subset $ Z$  of $\mba$ containing $M$, we write $\supp_Z u$ for  the {\bf support of $u$ in $Z$}, defined as the closure in $Z$ of $\supp\, u$.

	% Since the elements of  $W^{1,p}(\omd,\pa \omd)$ can  be extended by $0$ to elements of $W^{1,p}(\R^n)$ \cite[Proposition $4.5$]{lprime}, we will sometimes regard them as defined on $\R^n$. %We will write $W_0^{1,p}(\omd)$ for the closure of $\cc_0^\infty(\omd)$ in $W^{1,p}(\omd)$.
% 	Given a measurable function $\rho$ on $\omd$, we  set
% 	$$||u||_{L^p_\rho(\omd)}:=\left(\int_\omd |u|^p\rho\right)^{1/p}.$$

\item $\cc^{0,1}(A)$, space of Lipschitz functions on $A\subset \R^n$ (w.r.t. the euclidean norm).
 	\item 	%$\supp\, u$,  support of a distribution $u$ on $\omd$.
 We will regard $\cc^\infty(\mba)$ as a subset of $W^{1,p}(M)$.  In particular, for $u\in \cc^\infty(\mba)$,  $\supp\, u$ (the support of $u$) will be a subset of $M$ and
   we set for $Z$ open in $\mba$ $$\cc^\infty_{Z}(\mba):=\{u\in \cc^\infty(\overline{M}):\overline{ \supp\, u} \subset Z  \}.$$

% 	\item 	$\s_n$, set of globally subanalytic subsets of $\R^n$ (section \ref{sect_sub}).

%  \item 	$\D^+(S)$, set of positive definable continuous functions on $S\subset \R^n$ (section \ref{sect_sub}).

% 	\item 	$A_t$, fiber of $A\in \s_{m+n}$ at $\tim$ (section \ref{sect_pf_vshort}).

% 	  	\item 	$\cbf_t(A)$ (resp. $\cbf^S_t(A)$), tangent cone (resp. normal cone) of $A\subset \R^n$ at $t\in A$ (section \ref{sect_a_priori}, resp. section \ref{sect_normal_cone}). 	$\cbt_t(A)$ (resp. $\cbt^S_t(A)$) will stand for the link of the tangent (resp. normal) cone (section \ref{sect_a_priori}, resp. section \ref{sect_normal_cone}).

	\item 	$d(x,S)$,  euclidean distance from $x\in \R^n$ to  $S\subset \R^n$.% The function $x\mapsto d(x,S)$ is  denoted $d(\cdot, S)$.

	 	\item   If $S$ is a definable $\cc^\infty$ submanifold of $\R^n$ then there is a definable  neighborhood $U_S$ of $S$ on which  $\rho_S:x\mapsto d(x,S)^2$ is $\cc^\infty$ and there is a $\cc^\infty$ retraction $\pi_S:U_S\to S$ which assigns to every $x\in U_S$ the unique point that realizes $d(x,S)$, i.e. $|x-\pi_S(x)|=d(x,S)$ (see for instance \cite[Proposition $2.4.1$]{livre}).
	  We then set for $\mu$ positive continuous function on $S$
	  \begin{equation}\label{eq_tubular}
	  	U_S^\mu :=\{x\in U_S :d(x,S)<\mu( \pi_S(x))\}.
	  \end{equation}
We will also write $U_S^\mu$ when $\mu$ is a positive number, identifying it with the constant function. We then have for $x$ sufficiently close to $S$
\begin{equation}\label{eq_ker_dpi_nabla_drho}
\ker D_x \pi_S =T_{\pi_S(x)} S^\perp\; \et\; \nabla \rho_S(x)=2(x-\pi_S(x)).
\end{equation}

		\item $\hn^k$, $k$-dimensional Hausdorff measure.% and Hausdorff distance. and $d_\hn(\cdot, \cdot)$

	 	\item  Given two nonnegative functions $\xi$ and $\zeta$ on a set $E$ as well as a subset $Z$ of $E$, we write ``$\xi\lesssim \zeta$ on $Z$'' or  ``$\xi(x)\lesssim \zeta(x)$ for $x\in Z$'' when there is a constant $C$ such that $\xi(x) \le C\zeta(x)$ for all $x\in Z$. %We will however sometimes specify  constants, when useful.

% 	\item Given a  mapping $h:M\to M'$ between differentiable submanifolds,  $M\subset\R^n$ and $M'\subset \R^k$, differentiable at $x\in M$, with $\dim M \ge \dim M'$, we write $\jac h(x)$ for the (generalized) jacobian  $\sqrt{\det D_xh D_xh^\mathbf{t}}$, where $D_x h^\mathbf{t}$ denotes the transposite of $D_xh$. We   recall the  {\bf co-area formula}, which asserts that when
% 	 $h$ is Lipschitz and $m:=\dim M\ge m':=\dim M'$  we have for $u\in L^1(M)$:
% 	\begin{equation}\label{eq_coarea}
% 		\int_{y \in M'}\left( \int_{h^{-1}(y)}u(x)\, d\hn^{m-m'}(x) \,\right) d y  =\int_{x\in M } u(x)\,\jac h(x)\, d x .
% 	\end{equation}
%where .
	 %(defined at the points where $h$ is differentiable).

 	\item  We fix a decreasing  $\cc^\infty$ function $\psi:\R \to [0,1]$   
such that   \begin{equation}\label{eq_psi}
	 \psi\equiv 1 \quad \mbox{ on } \quad (-\infty,\frac{1}{2}), \quad\mbox{ and } \quad\psi\equiv 0\quad \mbox{ on } \quad[\frac{3}{4},\infty).
\end{equation}
	 \end{itemize}

\subsection{Subanalytic sets.}\label{sect_sub} We provide a few definitions in this section.	We refer the reader to  \cite{bm, ds, livre} for  basic facts about subanalytic geometry.

	\begin{dfn}\label{dfn_semianalytic}
		A subset $E\subset \R^n$ is called {\bf semi-analytic} if it is {\it locally}
		defined by finitely many real analytic equalities and inequalities. Namely, for each $a \in   \R^n$, there are
		a neighborhood $U$ of $a$ in $\R^n$, and real analytic  functions $f_{ij}, g_{ij}$ on $U$, where $i = 1, \dots, r, j = 1, \dots , s_i$, such that
		\begin{equation}\label{eq_definition_semi}
			E \cap   U = \bigcup _{i=1}^r\bigcap _{j=1} ^{s_i} \{x \in U : g_{ij}(x) > 0 \mbox{ and } f_{ij}(x) = 0\}.
		\end{equation}

		The flaw of the  semi-analytic category is that  it is not preserved by analytic morphisms, even when they are proper. To overcome this problem, we prefer working with the  subanalytic sets.

		A subset $E\subset \R^n$  is  {\bf  subanalytic} if
		each point $x\in\R^n$ has a neighborhood $U$ such that $U\cap E$ is the image under the canonical projection $\pi:\R^n\times\R^k\to\R^n$ of some relatively compact semi-analytic subset of $\R^n\times\R^k$ (where $k$ depends on $x$).

		A subset $Z$ of $\R^n$ is  {\bf globally subanalytic} if $\hh_n(Z)$ is   subanalytic, where $\hh_n : \R^n  \to (-1,1) ^n$ is the homeomorphism defined by $$\hh_n(x_1, \dots, x_n) :=  (\frac{x_1}{\sqrt{1+|x|^2}},\dots, \frac{x_n}{\sqrt{1+|x|^2}} ).$$

		We say that {\bf a mapping $f:A \to B$ is   globally subanalytic}, $A \subset \R^n$, $B\subset \R^m$ globally subanalytic, if its graph is a   globally subanalytic subset of $\R^{n+m}$. For simplicity, globally subanalytic sets and mappings will be referred as {\bf definable} sets and mappings.
We denote by $\s_n$ the set of definable subsets of $\R^n$.% and by $\D^+(S)$, $S\in \s_n$, the space of positive continuous definable functions on $S$.%  (this terminology is often used by o-minimal geometers \cite{vdd,cos})In the case $B=\R$, we say that  $f$ is a (resp. globally) {\bf  subanalytic function}.
\end{dfn}

The globally subanalytic category is well adapted to our purpose.  It is stable under intersection, union, complement, and projection. It thus constitutes an o-minimal structure \cite{vdd, cos} and consequently admits cell decompositions (see \cite[Definition $2.4$]{cos} \cite[Definition $1.2.1$]{livre}), from which it follows that definable sets enjoy a large number of finiteness properties (see \cite{cos,livre} for more).

Given $A\in \s_n$, we denote by $A_{reg}$  the set of points of $A$ at which $A$ is a $\cc^\infty$ submanifold of $\R^n$ of dimension $\dim A$. It follows from Tamm's theorem \cite{tamm} that this set is definable, and it easily follows from existence of cell decompositions \cite[Theorem $1.2.3$]{livre} that it  satisfies $\dim A \setminus A_{reg}<\dim A$.   We set  $\Gamma:=(\pa M)_{reg}$, where $\delta M :=\adh M\setminus M$.
% $$\pa \omd :=\{x\in \delta \omd:  \omd \mbox{ is a $\cc^\infty$ of $\R^n$ at $x$}  \}. $$that have a  neighborhood $U$ in $\R^n$ such that each connected component of $U\cap \omd$ is the interior of a  with boundary $U\cap \delta \omd$.
%$\delta \omd$
\begin{dfn}\label{dfn_stratifications}
	A {\bf
		stratification of}\index{stratification} a definable set $X\subset \R^n$ is a finite partition of it into
	definable $\cc^\infty$ submanifolds of $\R^n$, called {\bf strata}\index{stratum}. We say that a stratification $\Sigma$ of a set $X$ {\bf satisfies the frontier condition} if the closure in $X$ of every  $S\in \Sigma$ is a union of strata of $\Sigma$.\end{dfn}
Existence of stratifications satisfying the frontier condition follows from existence of Whitney $(b)$  regular  stratifications (see \cite[Proposition $2.6.17$]{livre}), which is well-known since \L ojasiewicz's work \cite{loj} (see also \cite[Propositions $2.6.6$ and $2.6.8$]{livre}).

We will say that $M$ is {\bf Lipschitz regular} at $x\in \delta M$ if this point has a neighborhood $U$ in $\R^n$ such that each connected component of $U\cap M$ is the interior of a Lipschitz manifold with boundary $U\cap \delta M$.
 We then let:
$$\pa M:=\{x\in \delta M: M\mbox{ is Lipschitz regular at $x$}\}.$$
It directly follows from existence of bi-Lipschitz trivial stratifications  \cite{pa,  lipsomin, halyin, paramreg} (or \cite[Corollary 3.2.17]{livre}) that this set is subanalytic. Moreover, we have
\begin{equation}\label{eq_pa1}\dim \left(\delta M \setminus  \pa M\right)\le m-2.\end{equation}
 This fact follows from a famous result  which is sometimes referred as {\it Wing Lemma} by geometers (see \cite[Proposition 1, section 19]{loj}, \cite[Proposition 9.6.13]{bcr}, or \cite[Lemma 5.6.7]{livre}).

 \begin{dfn}\label{dfn_embedded}
	We say that $M$ is {\bf connected at $x\in \delta M$}  if  $\bou(x,\ep)\cap M$ is connected for all  $\ep>0$ small enough. We say that $M$ is {\bf connected along $A\subset \delta M$} if  it is connected at every  $x\in A$, and that it is {\bf normal}   if it is connected along  $\delta M$.% (resp. {\bf weakly normal})(resp. $\hn^{m-1}$-almost every)
 \end{dfn}
	\section{Density of $\cc^\infty(\R^n)$}\label{sect_tra}
 At a point of $\pa M$, the germ of $M$ has finitely many connected components (bounded independently of the point, say at most $l$) which are Lipschitz manifolds with boundary.  Therefore, the restriction of every $u\in W^{1,p}(M)$, $p\in [1,\infty)$,  leaves a trace on $\pa M$, which is locally $L^p$, and we can define a trace operator $\tra_{\pa M}:W^{1,p}(M)\to L^p_{loc}(\pa M)^l$, which assigns to every $u$ its  traces provided by the respective restrictions of $u$ to the  local connected components of $M$ (when $M$ has less than $l$ connected components near a point, the other component functions of the trace are by convention $0$, see \cite{trace, lprime} for more).  This mapping indeed depends on the way the  connected components are enumerated, but we assume this choice to be made once for all, as a change of enumeration anyway just induces a permutation of the components of $\tra_{\pa M}$.
 Given an open subset $Z$ of $\pa M$, we then set
 $$W^{1,p}(M,Z):=\{u\in W^{1,p}(M):\tra_{\pa M} u =0 \mbox{ on } Z\}.$$
 It was shown in \cite[Theorem $4.1$]{lprime} that the elements of this space can be approximated by smooth functions (up to the boundary) vanishing in the vicinity of $\adh Z$ for $p$ not too big. More precisely:

\begin{thm}\label{thm_dense_lprime}Let $Z$ be a definable open subset of $\pa M$ and let $A$ be a definable subset of $ \delta M$ containing $\delta M\setminus \pa M$ and $\delta Z$, with   $\dim A\le m-2$. If $M$ is connected along $\pa M\setminus (Z \cup \adh A )$  then for all $p\in [1,m- \dim A]$ not infinite,    $\cc^\infty_{\mba\setminus (Z\cup \adh{A})}(\mba)$ is dense in $W^{1,p}(M,Z)$.
\end{thm}
Since one can always choose $A$ equal to the union of $\delta M\setminus \pa M$ and $\delta Z$, which are both of codimension $2$ in $\mba$, the above theorem ensures the density of $\cc_{\mba\setminus \adh{Z}}^\infty(\mba)$  in  $W^{1,2}(M,Z)$ as soon as $M$ is connected at $\hn^{m-1}$-almost every point of $\delta M\setminus \adh{Z}$, which is useful to produce weak formulations of PDEs. It is also worthy of notice that whenever $\dim \delta M \setminus \pa M\le m-3$, we obtain the density of $\cc^\infty(\mba)$ in $W^{1,p}(M)$ for all $p\in [1,3]$.
The following example shows that  this  is no longer true  for $p>3$.
\begin{exa}\label{exa_petit}
Let $k$ be an integer bigger than $1$. For $z\in (0,1)$, set
$$W_z:=(-2z,2z)\times (0,z^k)\subset \R^2$$
and
$$V_z:= \bou((-2z,0),z)\cup W_z\cup \bou((2z,0),z),$$
as well as (see picture)
$$\omd:= \{(x,y,z) \in \R^3: 0< z< 1, (x,y)\in V_z\}. $$

  \begin{figure}[h]
\includegraphics[ scale=0.4]{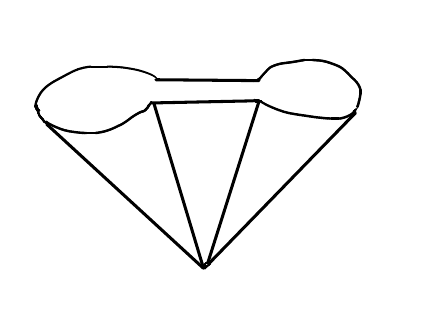}
\centering
  \caption{The germ of $\omd\subset \R^3$ at the origin. }%width=5cm, height=3cm,
\label{fig1}\end{figure}

Let for $(x,y,z)$ in $\omd$, $u(x,y,z):= \psi(\frac{x}{z})$, where $\psi$ is the function appearing in (\ref{eq_psi}). Since for all  $(x,y,z)$ in $\omd$, we have
\begin{equation}\label{eq_nabla_u}|\nabla u(x,y,z)| \lesssim |( \frac{1}{z},0,\frac{x}{z^2})|\lesssim \frac{1}{z},\end{equation}
and since in addition $\nabla u(x,y,z)=0$ for all $(x,y,z)\in\omd$ satisfying $(x,y)\notin  W_z$, we can write
$$\int_\omd |\nabla u|^p =  \int_0^1 \left(\int_{W_z} |\nabla u(x,y,z) |^pdxdy\right) dz\overset{(\ref{eq_nabla_u})}\lesssim \int_0^1 \frac{\hn^2(W_z)}{z^p} dz\lesssim \int_0^1 z^{k+1-p} dz, $$
which is finite if $p<k+2$. Hence, $u\in W^{1,p}(\omd)$ for such $p$.

Since $V_z$ is connected, $\omd$ is normal. We claim that $u$ cannot be approximated by elements of $\cc^\infty(\adh \omd)$ in the $W^{1,p}(\omd)$ norm  for any $p\in (3,k+2)$.
Assume otherwise, i.e. suppose that there is $v_j\in \cc^\infty(\adh\omd)$ such that $v_j\to u$ in
$W^{1,p}(\omd)$ with $p\in (3,k+2)$.
 Observe first that $u$ is $1$ on
 $$C_-:=0_{\R^3} *( \bou((-2,0),1)\times\{1\}), $$
 ($0_{\R^3} *A$ denoting the cone over a set $A$ at the origin of $\R^3$) and $0$ on
 $$C_+:=0_{\R^3} * (\bou((2,0),1)\times\{1\}) .$$
These two cones being Lipschitz domains, Morrey's embedding entails  that $v_j(0)$ tends to the limits at the origin of $u_{|C_-}\equiv 1$ and $u_{|C_+}\equiv 0$, which is clearly impossible.

%   The domain $\omd$ that we provided does not have isolated singularities. However, if we take $\omd':= \{(x,y,z)\in \omd: \rho(x,y,z)> \ep z^N \},$, where $\rho$ is a positive $\cc^\infty$ definable function vanishing on $\delta \omd$, $\ep>0$ small, $N\in \N$ large, we get an example which only has only a singularity at the origin and has the same properties (if $\ep>0$ is sufficiently small, by Sard Theorem,  it will be a regular value of $\frac{\rho}{z^N}$, so that the level $\{\frac{\rho}{z^N}=\ep\}$ will be smooth).The obstruction thus rather comes for the Lipschitz geometry of the domain and the curvature of the boundary.

\end{exa}
The above example shows that it is fairly easy to produce subanalytic open sets for which the space of functions that are smooth up to the boundary fails to be dense  in $W^{1,p}$  for some values of $p$. The case ``$p$ large'' was  studied in \cite{trace}, where the following positive result was obtained:

\begin{thm}\label{thm_trace}
Assume that $M$ is normal and let $A$ be a subanalytic subset of $\delta M$. For all $p\in [1,\infty)$ sufficiently large, we have:
\begin{enumerate}[(i)]
\item
$\cc^\infty(\adh{M})$ is dense in $W^{1,p}(M)$.
\item The linear operator
\begin{equation*}\label{trace}
\cc^\infty(\adh{M})\ni \varphi \mapsto \varphi_{|A}\in L^p(A,\hn^k), \qquad k:=\dim A,
\end{equation*}
is bounded for $||\cdot ||_{W^{1,p}(M)}$ and thus extends to a mapping $\tra_A:W^{1,p}(M)\to L^p(A,\hn^k)$.
\item If  $\mathcal{S}$ is a stratification of $A$, then $\cc^\infty_{\adh{M}\setminus\adh{A}}(\mba)$ is a dense subspace of
$\bigcap\limits_{Y\in\mathcal{S}}\ker \tra_Y$.

\end{enumerate}

\end{thm}

In the above theorem, the expression  ``for $p\in [1,\infty)$ sufficiently large'' means that there is $p_0$ such that the claimed statement holds for all $p>p_0$ not infinite. This number $p_0$ depends on the geometry of the singularities of $\delta M$ and   not only on $\dim M$, as shown by the above example.

%%%%%%%%%%%%%%%%%%%%%%%%%%%%%%%%%%%%%%%

\section{Density of Neumann regular functions}
A vector $\nbf\in \R^n$ is said to be {\bf normal at $x\in \Gamma$} if it is orthogonal to $ T_x\Gamma$.
Given an open subset $V$ of $\mba$, we will say that a function $v:V\to \R$  is {\bf Neumann regular} if it is $\cc^\infty$ at points of $V\cap \Gamma$ and   $\nabla v(x) \cdot \nbf=0$, for every $x\in V\cap \Gamma$ and $\nbf\in \R^n$ normal at $x$. We are going to derive from the density results presented in the previous section the density of Neumann regular functions.

\subsection{The case $p$ small} We start with a density theorem that covers the case $p=2$ and hence yields Corollary \ref{cor_2}.

\begin{thm}\label{thm_p_petit}
 Let $Z$ be an open definable subset of $\pa M$ and let $A$ be a definable subset of $\delta M$ containing $\delta M\setminus \Gamma$ and $\delta Z$,  and satisfying $\dim A\le m-2$. For every $p\in [1,m-\dim A]$, the space
 $$\{u\in \cc^\infty_{\mba\setminus (Z\cup \adh{A})} (\mba) : u \mbox{ is Neumann regular}  \}$$
 is dense in $W^{1,p}(M,Z)$ if and only if $M$ is connected along $\Gamma \setminus (Z\cup \adh{A})$.
\end{thm}
\begin{proof}
Assume first that $M$ fails to be connected at some $\xo\in \Gamma \setminus (Z\cup \adh{A})$.  The germ of $M$ at $\xo$ is therefore the union of (at least two) $\cc^\infty$ manifolds with common boundary $\Gamma$. Let $\phi$ be a function which coincides with a compactly supported $\cc^\infty$ bump function on one of them and which is zero on the others. The restriction of $\phi$ to each connected component inducing different traces on an open subset of $\Gamma$, it cannot be approximated in $W^{1,p}(M)$ (for any $p$) by smooth functions on $\mba$ (which would have the same trace) by the usual trace theorem (for smooth manifolds with boundary).

To prove the converse, assume  that $M$ is connected along $\Gamma\setminus (Z\cup \adh{A})$, which entails that it is locally the interior of a smooth manifold with boundary at every  point of $\Gamma\setminus (Z\cup \adh{A})$. Fix $p\in [1,m-\dim A]$.  Let $U_\Gamma$  be a sufficiently small definable neighborhood of $\Gamma$ for the closest point retraction $\pi_{\Gamma}$ (onto $\Gamma$) to be defined and smooth.
By Theorem \ref{thm_dense_lprime}, it suffices to show that every $u\in \cc^\infty (\mba)$ vanishing near $Z\cup \adh{A}$ can be approximated arbitrarily close by $\cc^\infty$ Neumann regular functions also vanishing  near $Z\cup \adh{A}$. For such a function $u$, set for $x\in U_{\Gamma}$
$$v(x):= u(\pi_{\Gamma}(x)) .$$
 This function is $\cc^\infty$ and vanishes near $Z\cup \adh{A}$. Moreover, for every $x\in \Gamma$ and each normal vector $\nbf$ at $x$, we have
 $$D_x v(\nbf) = D_{\pi_{\Gamma}(x)} u \,D_x\pi_{\Gamma} (\nbf)\overset{(\ref{eq_tubular})}{=}0, $$
 which means that $v:U_\Gamma\to \R$ is Neumann regular.    We now need to extend $v$ to a neighborhood of $\delta M$.

%	 Set for simplicity $Y:=\delta M\setminus \Gamma$.
 We claim that for any sufficiently small positive real number $\ep$, the function $v$ is zero on 
 $$V_\ep:=\{x\in U_{\Gamma}^\ep: d(x,\delta M\setminus \Gamma)<\ep \}$$
 (see (\ref{eq_tubular}) for $U_\Gamma^\ep$).
 Indeed, for $x$ in this set
 $$d(\pi_{\Gamma}(x),\delta M\setminus \Gamma)\le |x-\pi_{\Gamma}(x)|+d(x,\delta M\setminus \Gamma)<2\ep.$$
As $u$ is zero near $Z\cup \adh{A}\supset \delta M\setminus \Gamma$, it  must vanish on $\{x\in U_\Gamma: d(x,\delta M\setminus \Gamma)<2\ep\}$ for $\ep>0$ small, which, by the above inequality, means that  $v$ is zero on $V_\ep$ for such $\ep$, as claimed.

We thus now may smoothly extend $v$   to $$U^\ep_{\Gamma}\cup \{x\in \R^n: d(x,\delta M\setminus \Gamma)<\ep\} $$ by setting $\tilde{v}(x):=v(x)$ on $U_\Gamma^\ep$ and $\tilde{v}(x):=0$ if $d(x,\delta M\setminus \Gamma)<\ep$. Note that  the just above displayed set  is a neighborhood of $\delta M$. Let $\lambda$ be a smooth compactly supported function on this neighborhood, equal to $1$ on a smaller neighborhood of $\delta M$, and set
$w:=\lambda \cdot \tilde{v}. $
Because $w=v$ near $\Gamma$, this function is Neumann regular. It therefore suffices to approximate  $(u-w)$  by Neumann regular  functions (the sum of two  Neumann regular functions is clearly Neumann regular).
But since $(u-w)$ has trace zero and $A$ contains $ \delta M\setminus \pa M$,  Theorem \ref{thm_dense_lprime} ensures  that this function can be approximated for each $p\in [1,m-\dim A]$ in the $W^{1,p}(M)$ norm by elements of $\cc_0^\infty(M)$.
\end{proof}

We gave a fairly general statement, the most interesting case is when $p=2$, where we get the density of Neumann regular smooth functions as soon as $M$ is connected at $\hn^{m-1}$ almost every point of $\delta M$ (Corollary \ref{cor_2}). This is obtained by applying the above theorem with $A=\delta M\setminus \Gamma$, which has codimension at least $2$ in $\mba$.

\begin{rem}
  When outward cusps arise, the trace of an element of $W^{1,p}(M)$ may fail to be $L^p$, for $p$ small. If $u\in W^{1,p}(M)$ is a function that has an $L^p$ trace, it is possible to see that the smooth approximations of $u$ provided by the proof of Theorem \ref{thm_p_petit} can be required to have a trace tending to the trace of $u$ in the $L^p$ norm. This is because the proof of Theorem \ref{thm_dense_lprime} indeed  gives such approximations.
\end{rem}

%%%%%%%%%%%%%%%%%%%%%%%%%%%%%%%%%%%%%%%%%%%%%%%%%%%%%%%%%%%%

\subsection{The case $p$ large.} Example \ref{exa_petit} shows that the conclusion of Theorem \ref{thm_p_petit} does not hold for all $p$. We establish in this section the density of {\it Lipschitz} Neumann regular functions (smooth on $M\cup \Gamma$) in  the case ``$p$ sufficiently large'', i.e. for $p$ bigger than some $p_0$ (which depends on $M$), when $M$ is connected at boundary points. We then provide an example  to show that the density of $\cc^2(\mba)$ Neumann regular functions may fail for arbitrary large values of $p$ (Example \ref{exa_pgd}).

\begin{thm}\label{d}
For every sufficiently large value of $p$, the space
$$\{u\in \cc^{0,1}(\mba)\cap \cc^\infty(M\cup \Gamma) :u \mbox{ is Neumann regular}\}$$
is dense in $W^{1,p}(M)$ if and only if $M$ is normal.
\end{thm}

The proof requires some material. We first prove a lemma that gives a Lipschitz tubular neighborhood near a given point $\xo \in \delta \Gamma$ and then construct  Neumann regular partitions of unity. This tubular neighborhood will be used  to produce our partitions of unity as well as to prove Theorem \ref{d}.

We recall that a  {\bf tubular neighborhood} of a manifold $S$  is a triple $(U,\pi,\rho)$, with $U$ neighborhood of $S$, $\pi: U\to S$   retraction,  and $\rho: U\to \R$ nonnegative function satisfying $\rho^{-1}(0)=S$. We say that a tubular neighborhood is $\cc^k$ or Lipschitz on a set if so are both $\pi$ and $\rho$ on this set.

\begin{lem}\label{lem_t}
 Let  $S\subset \delta  \Gamma$ be  a $\cc^\infty$  definable submanifold of $\R^n$. For each $x_0\in S$ there is  a Lipschitz   tubular neighborhood $(U_{\xo},\pi,\rho)$ of $S\cap U_\xo$ which is $\cc^\infty$ on $U_\xo\setminus \delta \Gamma$ and satisfies near $\Gamma$:
 \begin{equation}\label{eq_comp_cond}
\rho \circ \pi_{\Gamma}=\rho \et \pi \circ \pi_{\Gamma}=\pi.
 \end{equation}\end{lem}

\begin{proof} Let $U_\Gamma$  be a sufficiently small definable neighborhood of $\Gamma$ for the closest point retraction $\pi_{\Gamma}$ (onto $\Gamma$) to be smooth. Taking $U_\Gamma$ smaller if necessary, we may assume  $|D_x\pi_\Gamma|$ to be bounded on $U_\Gamma$ (it is bounded on $\Gamma$ in virtue of (\ref{eq_ker_dpi_nabla_drho})). %$U_S$ being the domain of a local retraction $\pi_S$ ,
Let for $x\in \Gamma$
$$\mu(x):= \min (\frac{1}{2}d(x,\R^n\setminus U_\Gamma),d(x,\delta \Gamma)) .$$
 Since $\mu(x)<d(x,\R^n\setminus U_\Gamma) $, the topological frontier of $U^\mu_\Gamma$ in $U_\Gamma$ (see (\ref{eq_tubular}) for $U^\mu_\Gamma$) coincides with $\{x\in U_\Gamma:d(x, \Gamma)=\mu (\pi_\Gamma(x))\}$. Choose a positive $\cc^\infty$ function $\sigma<\mu$ on $\Gamma$  with bounded first derivative (if we choose  an  exhaustive sequence $(K_i)_{i\in \N}$ of compact subsets of $\Gamma$, a smooth partition of unity $\varphi_i$ subordinate to $int(K_{i+2})\setminus K_{i}$, and a sequence of positive numbers $(\ep_i)_{i\in \N}$, the function $\sigma:=\sum \ep_i \varphi_i$ will have the required  properties if   $(\ep_i)_{i\in \N}$  is decreasing sufficiently fast).
% Up to a partition of unity we can assume that $u$ is supported in some $U_{x_i}$, say $U_{x_0}$.   %The result follows by induction.

Set for $x\in U_\Gamma^\mu$ close to $\xo$
$$\alpha(x):=\psi\left( \frac{\rho_{\Gamma}(x)}{\sigma(\pi_{\Gamma}(x))^2}\right),$$
where $\psi$ is as in (\ref{eq_psi}).
 Notice that                                                                                             on the support of  $\psi'(\frac{\rho_{\Gamma}(x)}{\sigma(\pi_{\Gamma}(x))^2})$, we have
 \begin{equation}\label{eq_rhosigma}
 \rho_{\Gamma}(x)<\sigma(\pi_{\Gamma}(x))^2,
 \end{equation}
 which entails $|D_x \rho_{\Gamma}|\overset{(\ref{eq_ker_dpi_nabla_drho})}\le 2\sigma(\pi_\Gamma(x))$ on this set, and therefore:
\begin{equation}\label{eq_bound_Dalpha}
	|D_x\alpha|= |\psi'(\frac{\rho_{\Gamma}(x)}{\sigma(\pi_{\Gamma}(x))^2})||\frac{D_x\rho_{\Gamma}}{\sigma(\pi_{\Gamma}(x))^2}-\frac{2\rho_{\Gamma}(x) D_{\pi_{\Gamma}(x)}\sigma \cdot D_x\pi_{\Gamma}}{\sigma(\pi_{\Gamma}(x))^3}|\lesssim \frac{1}{\sigma(\pi_{\Gamma}(x))}.\end{equation}
We now can set for $x\in U_\Gamma ^\mu$ close to $\xo$:
$$\rho(x):=\alpha(x)\rho_S(\pi_{\Gamma} (x))+(1-\alpha(x))\rho_S(x).$$
As $\rho_\Gamma(x)=\mu(\pi_{\Gamma}(x))^2$ on  the topological frontier of $U^\mu_\Gamma$ in $U_\Gamma$ and since $\sigma<\mu$, the function $\alpha$ vanishes near the topological frontier of $U^\mu_\Gamma$ in $U_\Gamma$,  which means that
$\rho:=\rho_S$ extends
 continuously $\rho$ to a neighborhood of $\xo$ (since $\mu(x)\le d(x,\delta \Gamma)$, $\rho_S(\pi_\Gamma(x))$ tends to $\rho_S(x)$ as $x\in U_\Gamma^\mu$ tends to $\delta \Gamma$), smoothly at every point of $M\cup \Gamma$.
Since $\alpha\equiv 1$ near $\Gamma$, the condition $\rho\circ \pi_\Gamma=\rho$ holds by construction.
Moreover, as for $x$ close to $\xo$ and $w\in \R^n$ we have
$$D_x\rho(w)=\alpha(x)D_{\pi_{\Gamma}(x)}\rho_S D_x\pi_{\Gamma}(w) +(1-\alpha(x))D_x\rho_S(w)+ (\rho_S(\pi_{\Gamma} (x))- \rho_S(x))D_x\alpha(w),$$
and since 
$$|\rho_S(\pi_{\Gamma} (x))- \rho_S(x)|\le |x-\pi_\Gamma(x)|=\rho_S(x)^{1/2}\overset{(\ref{eq_rhosigma})}< \sigma(x) ,$$
(\ref{eq_bound_Dalpha}) entails that  $|D_x\rho|$ is bounded near $\xo$ on $U_\Gamma^\mu$. On $\bou(\xo,\ep)\setminus U_\Gamma^\mu$, $\ep >0$ small, as $\rho=\rho_S$,  $|D_x\rho|$ is also bounded  on the complement of $\delta \Gamma$.  Hence, $|D_x \rho|$  is bounded on $\bou(\xo,\ep)\setminus\delta \Gamma$.

 Because $\delta \Gamma$ is subanalytic and has empty interior, almost every segment of line cuts this set at finitely many points (possibly none). Since the restriction of $\rho$ to every such segment has bounded derivative on the complement of a finite set, and since $\rho$ is continuous,    we  see that $\rho$ is $L$-Lipschitz near $\xo$ with $L:= \sup\{ |D_x\rho|:x\in \bou(\xo,\ep)\setminus \delta \Gamma\}$, $\ep>0$ small.

To construct  $\pi$, we will argue up to a local coordinate system $\phi:U_\xo\to U_0$ of $S$ at $\xo$ (we recall that the strata are $\cc^\infty$), where $U_0$ (resp. $U_\xo$) is a neighborhood of the origin (resp. of $\xo$) in $\R^n$. We will thus identify $\xo$ with $0$, $S$ with a neighborhood  of the origin in $\R^k\times \{0_{\R^{n-k}}\}$, where $k=\dim S$, $M$ and $\Gamma$ with their respective images under $\phi$, and $ \pi_{\Gamma}$ with $\phi\circ \pi_{\Gamma} \circ \phi^{-1}$.
Set then for $x\in U^\mu_{\Gamma}$, with $\mu$ as above,
$$ \pi(x):=\alpha(x)\pi_k (\pi_{\Gamma}(x))+ (1-\alpha(x))\pi_k(x),$$
where $\alpha$ is as above and $\pi_k:\R^n\to\R^k\times \{0_{\R^{n-k}}\} $ is the orthogonal projection.
Setting $\pi(x)=\pi_k(x)$ for $x\in U_0\setminus U_{\Gamma}^\mu$ extends continuously $\pi$, smoothly on the complement of $\delta \Gamma$, and since $\alpha\equiv 1$ near $\Gamma$, we clearly have $\pi\circ \pi_\Gamma=\pi$ near $\Gamma$.
Moreover, since for $w\in \R^n$ and $x\in U_\Gamma^\mu$ close to $0$, we have
$$D_x\pi(w)= \alpha(x)D_{\pi_{\Gamma}(x)}\pi_k D_x\pi_{\Gamma}(w)+(1-\alpha(x)) D_x\pi_k(w)+ D_x\alpha( w)\cdot (\pi_k (\pi_{\Gamma}(x))-\pi_k(x)),$$
and because, by definition of $\pi_k$ and $\pi_\Gamma$, we have
$$|\pi_k (\pi_{\Gamma}(x))-\pi_k(x)|\le |\pi_{\Gamma}(x)-x|\lesssim \sqrt{\rho_{\Gamma}(x)}, $$
we see that (\ref{eq_rhosigma}) and (\ref{eq_bound_Dalpha}) entail that $|D_x\pi|$ is bounded near  $0$. Since $\pi$ is continuous, smooth  on the complement of $\delta  \Gamma$ in some neighborhood of $0$, and has bounded derivative, the same argument as for $\rho$ shows that it is a Lipschitz mapping.
\end{proof}

%
%  \begin{enumerate}
%   \item $\rho \circ \pi_{\pa M}=\rho$ and $\pi \circ \pi_{\pa M}=\pi$ on $U_{\pa M}\cap U_\xo$.
% \item
%  \end{enumerate}

We now are going to construct partitions of unity by Neumann regular functions,  which can also be useful for other applications. Like the above tubular neighborhood, these partitions of unity will just be Lipschitz, but Example \ref{exa_pgd} shows that it is not possible to get something much better.

\begin{pro}\label{pro_p}
Given a finite covering of $\mba$ by open sets, there is a $\cc^{0,1}(\mba)$ partition of unity subordinate to this covering, $\cc^\infty$ on $M\cup \Gamma$, and which is exclusively constituted by Neumann regular functions.
\end{pro}
\begin{proof}
It suffices to construct Lipschitz Neumann regular bump functions, i.e. to prove that every $\xo\in  \mba$ has a fundamental system of neighborhoods  on each of  which we can find a  Neumann regular compactly supported Lipschitz function $\theta_\xo$ ($\cc^\infty$ on $M\cup \Gamma$) that satisfies $\theta_\xo\equiv 1$ in the vicinity of $\xo$.   Fix $\xo\in \adh{\Gamma}$ (at points of $\mba\setminus \adh{\Gamma}$ the statement is obvious) and a number $\nu>0$.

\underline{First Case:} $\xo \in \delta \Gamma$. Let $(U_\xo,\pi,\rho)$ denote the (local) Lipschitz tubular neighborhood provided by Lemma \ref{lem_t}, applied with $S=\{\xo\}$.  Set
$$\theta_\xo(x):= \psi_\nu(\rho(x)),  $$
where $\psi_\nu(t):=\psi(\frac{t}{\nu})$, $\psi$ being as in (\ref{eq_psi}). The function $\theta_\xo$ is Lipschitz (resp. $ \cc^\infty$ on $M\cup\Gamma$) as a composite of Lipschitz (resp. $\cc^\infty$) functions. It is $1$ on a neighborhood of $\xo$ and  is supported in $\{x\in U_\xo:  \rho(x)\le \frac{3}{4}\nu\}$.
Furthermore, for each $x\in  \Gamma \cap U_\xo$ we have for every vector $\nbf$  normal at $x$
 $$D_x\theta_\xo(\nbf)=\psi_\nu'(\rho(x))D_x\rho(\nbf)\overset{(\ref{eq_comp_cond})}=\psi_\nu'(\rho(x))D_x\rho(D_x \pi_\Gamma (\nbf))\overset{(\ref{eq_tubular})}{=}0.$$
%  and since we have for $x\in \Gamma\cap U_\xo$ %$D_x \pi_{\Gamma} (\nbf)=0$, which entails
%  $$ D_x\rho(\nbf)\overset{(\ref{eq_comp_cond})}=D_x\rho(D_x \pi_\Gamma (\nbf))\overset{(\ref{eq_tubular})}{=}0,$$
%   we see that $\theta_\xo$ is Neumann regular.
  %$D_x\theta_\xo(\nbf)=0$, as required.

 \underline{Second case:}  $\xo \in \Gamma$. In this case, define a $\cc^\infty$ function by
 $$\theta_\xo(x):=\psi_\nu(\eta(x))\cdot \psi_\nu(\rho_\Gamma(x)),
 $$
where $\eta(x):=|\pi_\Gamma(x)-\xo|^2$ and $\psi_\nu$ is as above. A computation of derivative yields for every vector $\nbf$  normal at  $x\in  \Gamma \cap U_\xo$
  $$D_x\theta_\xo(\nbf)=\psi_\nu(\rho_{\Gamma}(x)) \psi_\nu'(\eta(x)) D_x\eta( \nbf) +\psi_\nu(\eta(x)) \psi_\nu'(\rho_{\Gamma}(x))D_x\rho_{\Gamma}(\nbf),$$
  and since  we have
  $$D_x\eta(\nbf)=2(\pi_{\Gamma}(x)-\xo)\cdot D_x\pi_{\Gamma}(\nbf)\overset{(\ref{eq_tubular})}{=}0$$
  and $\psi'_\nu (0)=0$,
   we see that the Neumann condition holds.
\end{proof}%with  Namely, we set $$\theta_\xo:=\psi(\pi(x))\cdot \phi(\rho^2(x)),$$
%  where $\pi$ and $\rho$ are provided  by the latter lemma. We now put a square to $\rho_\Gamma$ because we wish $\theta_\xo$ to be smooth (it is to be smooth at points of $\Gamma$) and $\rho_\Gamma$ is not smooth while $\rho_\Gamma^2$  is

\begin{proof}[Proof of Theorem \ref{d}.]
Assume that $M$ is not connected at a point $\xo\in \delta M$ and denote by $U$ a  connected component of $M\cap \bou(\xo,\ep)$, $\ep>0$ small. Take a function which coincides with a compactly supported $\cc^\infty$ bump function on $U$ and which is zero on the other connected components of $M\cap \bou(\xo,\ep)$.  Such a function belongs to $W^{1,\infty}(M)$ but does not extend continuously at $\xo$, and hence cannot be approximated by Lipschitz functions, by \cite[Corollary $3.7$]{trace} (applied with $A:=\{\xo\}$, see the proof of \cite[Corollary $3.9$]{trace}  for more, the  argument is completely analogous). Hence, the space of Lipschitz functions  fails to be dense in $W^{1,p}(M)$ for arbitrarily large values of $p$ if $M$ is not connected at every boundary point.

To prove the converse, suppose that  $M$ is normal.
Let $\Sigma$ be a stratification of $\delta M$ satisfying the frontier condition and compatible with $\Gamma$ (i.e.  $\Gamma$ is a union of strata). The idea of the proof is to first approximate the given function near the strata of dimension $0$ and next approximate the upshot successively near the strata of higher dimension. We therefore naturally introduce the following function $\kappa$ on which our induction will rely.

 Given $j\le m-1$, we denote by $X_j$ the union of all the strata of
 $\Sigma$ of dimension less than or equal to $j$, with $X_{-1}=\emptyset$. Given a function $u$  on an open subset $V$  of $ \mba$, we set
$$\kappa(u):=\max \{k \le m: u \mbox{ is Neumann regular on a neighborhood of  } X_{k-1}\cap V\mbox{ in } \mba \},$$
and we will prove by decreasing induction on $k\le m$ the following statements:

$\mathbf{(H_k)}$ For $p$ sufficiently large, given $u\in W^{1,p}(M)$  satisfying $\kappa(u)\ge k$, there is a sequence of Lipschitz functions $(v_j)_{j\in \N}$ in $  \cc^\infty(M\cup \Gamma)$ tending to $u$ in $W^{1,p}(M)$ and exclusively constituted by Neumann regular functions.

The theorem follows from $ \mathbf{(H_0)}$. Let us first check $\mathbf{(H_m)}$. Given $u$ such that $\kappa(u)=m$, there is a neighborhood $W$ of $\delta M$ on which $u$ is Neumann regular. Moreover, as well-known, $\cc^\infty(M)\cap W^{1,p}(M)$ is dense in $W^{1,p}(M)$. It then suffices to glue $u_{|W}$ with the terms of a sequence of smooth approximations of $u_{|M}$ by means of a partition of unity which is constant on some neighborhood of $\delta M$.

  We  now are going to prove $\mathbf{(H_k)}$, for $k\le m-1$, assuming $\mathbf{(H_{k+1})}$. Fix  $u\in W^{1,p}(M)$ satisfying $\kappa(u)=k$.  Note that if two functions $v$ and $w$ are Neumann regular in the vicinity of a point $x$, so are $ vw$ and $v+w$. As a matter of fact, since Proposition \ref{pro_p} provides us Neumann regular partitions of unity, the induction step reduces to prove the  following local fact for $p$ large enough: 

\noindent{\bf Claim.} Every $\xo\in \delta M$ has an open neighborhood $V_\xo$ in $\mba$ for which there is a sequence $(v_j)_{j\in \N}$ in  $\cc^{0,1}(V_\xo)$  tending to $u$ in $W^{1,p}(V_\xo\cap M)$ and satisfying $\kappa(v_j)> k$ for all $j$.
%, smooth at points of $\Gamma$,

Fix $\xo\in \delta M$ and denote by $S$ the stratum of $\Sigma$ that contains $\xo$. If $S\subset \delta \Gamma$, let $( U_\xo,\pi,\rho)$ be a local tubular neighborhood of $S$ as provided by Lemma \ref{lem_t}. If  $S\subset  \Gamma$, we just let $\pi:=\pi_{\Gamma}$ and define $U_\xo$ to be a small ball centered at $\xo$. Thanks to the frontier condition, choosing $U_\xo$ smaller if necessary, we can require that $S$ is the only stratum  of dimension less than or equal to $  \dim S$ met by  $U_\xo$.

If $\dim S<k$ then $\xo \in X_{k-1}$, which, since $\kappa(u)=k$, means that $\xo$ has a neighborhood $V_\xo$ on which $u$ is Neumann regular. Hence, we may set in this case $v_j:=u_{|V_\xo}$ for all $j$.

Assume thus that $\dim S\ge k$. Since $M$ is normal,
by Theorem \ref{thm_trace}, for $p$ large enough there is a sequence $g_j\in \cc^\infty(\mba)$ such that for each integer  $j\ge 1$
\begin{equation}\label{fp}||u-g_j||_{W^{1,p}(M)}<\frac{1}{j}.\end{equation}                                                                                                     Let $\lambda $ be a  $\cc_0^\infty$ function on $U_\xo\cap \mba$ equal to $1$ on some neighborhood $V_\xo$ of $\xo$ in $\mba$. Note that since $\lambda\cdot (g_j\circ \pi-g_j)$ has trace zero on $S$, Theorem \ref{thm_trace} entails (for $p$ large enough) that there is for each integer $j\ge 1$ a function  $f_j\in\cc^\infty_{\mba \setminus \overline{S}}(\mba)$  such that
$$||f_j-\lambda\cdot (g_j-g_j\circ \pi)||_{W^{1,p}(M)}<\frac{1}{j},$$
which implies  that
\begin{equation}\label{fgp}||f_j-(g_j-g_j\circ \pi)||_{W^{1,p}(V_\xo\cap M)}<\frac{1}{j}.\end{equation}

Let us now check that $v_j:=f_j+g_j\circ \pi$ (on $V_\xo$) has all the required properties. It is a Lipschitz function which is smooth everywhere $\pi$ is. Moreover,
$$||v_j-u||_{W^{1,p}(V_\xo\cap M)}\le ||(v_j-g_j\circ \pi)-(g_j-g_j\circ \pi)||_{W^{1,p}(V_\xo\cap M)}+||g_j-u||_{W^{1,p}(V_\xo\cap M)},$$
which,  by (\ref{fp}) and (\ref{fgp}), yields that $v_j$ tends to $u$. Since $U_\xo$ (and hence $V_\xo$) does not meet any other stratum of dimension $\le k$  than $S$ and $f_j$ vanishes near $S$, we see that $\kappa(f_j)>k$.  Note also that if $S\subset \delta\Gamma$ then for every $x\in V_\xo \cap \Gamma$ and every vector $\nbf\in \R^n$ normal at $x$
  $$D_x (g_j\circ \pi)(\nbf)=D_{\pi(x)} g_j D_x\pi(\nbf)\overset{(\ref{eq_comp_cond})}=D_{\pi(x)} g_j D_{\pi_\Gamma(x)}\pi D_x\pi_{\Gamma}(\nbf)\overset{(\ref{eq_tubular})}{=}0,$$
 %Again, since $V_\xo$ does not meet any other stratum of dimension $\le k$  than $S$, this
 which yields $\kappa (g_j\circ \pi)=m>k$. Moreover, if $S\subset \Gamma$ then $\pi=\pi_{\Gamma}$ and this obviously continues to hold. We conclude that  $\kappa(v_j)>k$, as needed.
\end{proof}
We now provide a couple of examples that shows that neither the partition of unity given by Proposition \ref{pro_p} nor the approximations given by Theorem \ref{d} could be required to be $\cc^2$ at singular points of $\delta M$.

\begin{exa}\label{exa_pgd}
 Let $$\omd :=\{(x,y,z)\in \R^3: xy(y+x)(y-zx)>0, 0<z<1  \}.$$
 For every $z_0\in (0,1)$, the boundary of
 $$\omd_{z_0}:=\omd \cap \{z=z_0\}$$
 consists of four lines
  $x=0$, $y=0$, $y=-x$, $y=z_0 x$.  These families of lines constitute surfaces that are smooth outside the $z$-axis.  We thus can stratify $\adh{\omd}$ by $S:=\{0_{\R^2}\}\times (0,1)$, the complement of the $z$-axis in these surfaces, and  $\omd$. This example is unbounded but our problem will be local and we could cut it with a little ball near a point $(0,0,z_0), z_0\in (0,1)$.

  We claim that every $u\in \cc^2(\adh{\omd})$ satisfying Neumann's condition at the smooth points of $\delta \omd$  is constant on $S$. Assume otherwise, i.e. assume that $\frac{\pa u}{\pa z}(0,0,z_0)\ne 0$ for some $z_0\in (0,1)$, for such a function $u$. Set  for $(x,y,z)\in \omd \cap \bou((0,0,z_0),\mu),$ $\mu$ being a positive number sufficiently small for $\frac{\pa u}{\pa z}$ to be nonzero on this ball,
$$\eta(x,y,z):= \left(\frac{\pa u}{\pa z}(x,y,z)\right)^{-1}\cdot \nabla u(x,y,z).$$
 As $u$ satisfies Neumann's condition, $\eta$ is tangent to $\delta \omd$ at every smooth point. It is easily checked that by continuity it must be tangent to all the strata. This vector field therefore  generates a $\cc^1$ local flow $(q,t)\mapsto \phi_t(q)$, for $q=(x,y,z)$ close to $(0,0,z_0)$, that preserves  $\delta \omd$.

Denote by $P$ the canonical projection onto the $z$-axis.  Since $P(\eta)\equiv (0,0,1) $, we must have $P(\phi_t(x,y,z))=z+t$, which means that $\phi_t$ induces a diffeomorphism between the respective germs of  $\omd_{z_0}$ and  $\omd_{z_0+t}$ at the origin. It is well-known that this is impossible. The derivative at the origin of such a diffeomorphism would respectively send the four lines $x=0$, $y=0$, $y=-x$, $y=z_0 x$ onto the lines $x=0$, $y=0$, $y=-x$, $y=(z_0+t) x$. The vectors generating the first three lines would be eigenvectors of the restriction of $D_{(0,0,z_0)}\phi_t$ to $\R^2\times\{0\}$, which would thus be a homothetic transformation of $\R^2\times\{0\}$, which therefore could not send the last line onto a different one.

This counterexample is due to H. Whitney   and was historically one of the motivations for geometers to study singularities from the $\cc^0$ point of view.
It shows that the partition of unity provided by Proposition \ref{pro_p}, which is just Lipschitz, could not be required to be $\cc^2$ at singularities (it is however smooth at smooth points of the boundary). It is unclear to the author whether it could be required to be $\cc^1$.% but it seems that obstructions will arise in higher dimensions as the number of strata grows.

This example is not connected at the points of the $z$-axis and therefore does not show that Theorem \ref{d} could not provide $\cc^2$ approximations. It is however not difficult to produce an analogous example in $\R^4$ which would be connected at frontier points, as follows. Take a polygon $X_t$ in $\R^2$ with one moving vertex and the other fixed, say $A(t)=(0,t)$. Take then the cone in $\R^3$ over $X_t\times \{1\}$ at the origin  (see picture) and denote by $Y_t$ the interior of this cone. The domain $U:=\{(x,y,z,t)\in \R^4: (x,y,z)\in Y_t\}$   is connected at each frontier point.

  \begin{figure}[h]
\includegraphics[ scale=0.35]{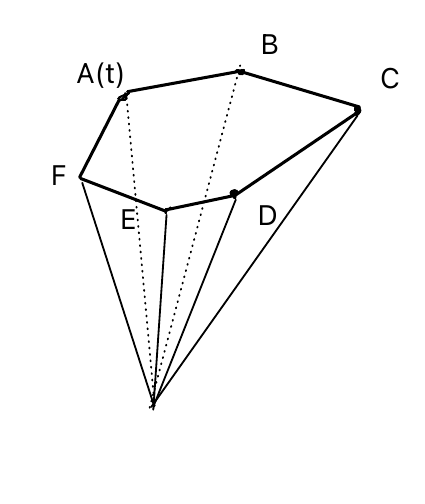}
\centering
  \caption{The set $Y_t$ is the interior of this polyhedron. }%width=5cm, height=3cm,
\label{fig1}\end{figure}
Any non constant $\cc^2$ function $u$ near the $t$-axis satisfying Neumann's condition at smooth boundary points will give rise to a $\cc^1$ flow mapping the cone over $X_t\times \{1\}$ onto the cone over $X_{t'}\times \{1\}$ for  $t\ne t'$, which will be impossible for the same reason as for $\omd_{z_0}$ and $\omd_{z_0+t}$. It is worthy of notice that this domain has Lipschitz boundary.\end{exa}

\begin{rem}\label{rem_concluding}
\begin{enumerate}
 \item Theorems \ref{thm_p_petit} and \ref{d} provide density results for Neumann regular functions under connectedness assumptions at boundary points that are proved to be necessary. If we drop these assumptions, the proofs yield the density of $W^{1,\infty}$ functions satisfying Neumann's condition.  Elements of $W^{1,\infty}$ are not smooth, they are only Lipschitz with respect to the inner metric (given by the lengths of paths), and  can have several limit values at a boundary point at which $M$ has several connected components.
%We also can  require (in both theorems)  the approximations that we construct  to be smooth at points of $\Gamma$ on the closure of each of these connected components.

\item The partition of unity given by Proposition \ref{pro_p} is not subanalytic. If we demand it to be subanalytic, it is no longer possible to require it to be $\cc^\infty$ at points of $\Gamma$. We however can  construct it $\cc^k$, with $k$ arbitrary large. The reason is that $\cc^\infty$ subanalytic functions being analytic, they must be zero on the whole of the connected component as soon as they vanish on a neighborhood of a point.

% \item 
% 
% Thanks to Proposition \ref{pro_p}, Theorem \ref{d} and \ref{thm_trace} establish together 
%  that if $M$ is normal and  $Z$ is a definable subset of $\pa M$ then $\cc^{0,1}(\mba)\cap \cc^\infty_{\mba\setminus \adh Z}(M\cup \Gamma)$ is dense  in $W^{1,p}(M,Z)$ for all $p$ sufficiently large.

\item The results of this article are indeed valid if $M$ is unbounded   since we could use cut-off functions to reduce the problem to a function with bounded support.

 \end{enumerate}

\end{rem}

	\end{document}